\documentclass{amsart}

\usepackage[a4paper,hmargin=3.5cm,vmargin=4cm]{geometry}
\usepackage{amsfonts,amssymb,amscd,amstext}
\usepackage{graphicx}
\usepackage[dvips]{epsfig}
\usepackage{mathpazo}

\renewcommand{\leq}{\leqslant}
\renewcommand{\geq}{\geqslant}

\newcommand{\hh}{{\mathbb{H}}}

\newtheorem{theoremA}{Theorem}

\renewcommand{\thetheoremName}

\newtheorem{theorem}{Theorem}[section]
\newtheorem{proposition}[theorem]{Proposition}
\newtheorem{lemma}[theorem]{Lemma}
\newtheorem{corollary}[theorem]{Corollary}

\theoremstyle{definition}
\newtheorem{definition}[theorem]{Definition}
\newtheorem{remark}[theorem]{Remark}

\theoremstyle{remark}

\numberwithin{equation}{section}

\newcommand{\dist}{\operatorname{dist}}
\newcommand{\Vol}{\operatorname{Vol}}

\newcommand{\Div}{\operatorname{div}}

\newcommand{\kan}{\mathbb{K}^{n}(b)}
\newcommand{\kam}{\mathbb{K}^{m}(b)}
\newcommand{\erre}{\mathbb{R}}
\newcommand{\ene}{\mathbb{N}}

\numberwithin{equation}{section}

\def\cprime{$'$} \def\cprime{$'$}
\providecommand{\bysame}{\leavevmode\hbox
to3em{\hrulefill}\thinspace}
\providecommand{\MR}{\relax\ifhmode\unskip\space\fi MR }
 \providecommand{\href}[2]{#2}

\begin{document}

\title[Moment Spectra and First Eigenvalues]
{Estimates of the first Dirichlet eigenvalue \\ from exit time moment spectra}

\author[A. Hurtado]{A. Hurtado$^{\natural}$}
\address{Departamento de Geometr\'{\i}a y Topolog\'{\i}a, Universidad de Granada, E-18071,
Spain.}
 \email{ahurtado@ugr.es}
\author[S. Markvorsen]{S. Markvorsen$^{\#}$}
\address{DTU Compute, Mathematics, DK-2800 Kgs. Lyngby, Denmark}
\email{stema@dtu.dk}
\author[V. Palmer]{V. Palmer*}
\address{Departament de Matem\`{a}tiques-INIT, Universitat Jaume I, Castell\'o,
Spain.} \email{palmer@mat.uji.es}
\thanks{$^{\natural}$  Supported by the Spanish Mineco-FEDER grant
MTM2010-21206-C02-01 and Junta de Andalucia grants FQM-325 and P09-FQM-5088.\\ $^{\#}$ \, $^{*}$    Supported by the Spanish Mineco-FEDER grant
MTM2010-21206-C02-02 and by the Pla de Promoci\'o de la Investigaci\'o de la Universitat Jaume I}
\subjclass[2000]{Primary 58C40; Secondary 53C20}
%

\keywords{Riemannian submanifolds, extrinsic balls, torsional
rigidity, $L^1$-moment spectrum, mean exit time, isoperimetric
inequalities, Green operator, Poincar\'{e} and Barta inequalities}

\begin{abstract}
We compute the first Dirichlet eigenvalue of a geodesic ball in a rotationally symmetric model space in terms of the moment spectrum for the Brownian motion exit times from the ball. As an application of the model space theory we prove lower and upper bounds for the first Dirichlet eigenvalues of extrinsic metric balls in submanifolds of ambient Riemannian spaces which have model space controlled curvatures. Moreover, from this general setting we thereby obtain new generalizations of the classical and celebrated results due to McKean and Cheung--Leung concerning the fundamental tones of Cartan-Hadamard manifolds and the fundamental tones of submanifolds with bounded mean curvature in hyperbolic spaces, respectively.
\end{abstract}

\maketitle

\section{Introduction}\label{secIntro}
\bigskip

Given  a complete Riemannian manifold $(N^n, g)$ and a normal domain $D \subset N$, we denote by $\lambda_1(D)$ the first eigenvalue of the Dirichlet boundary value problem

\begin{equation}
\begin{aligned}
\Delta u+\lambda u &= 0\,\,\, \text{on}\,\,\, D\\
u\vert_{\partial D} &=0 \quad,
\end{aligned}
\end{equation}

\noindent where $\Delta$ denotes the Laplace-Beltrami operator on $\,(N^{n},
g)\,$. The increasing sequence of eigenvalues $\{\lambda_k(D)\}_{k=1}^\infty$ for this problem  is the {\em Dirichlet spectrum} of $D$.

When we consider $D=B^N(p,r)$, the geodesic $r$-balls centered at a given point $p \in N$, the first eigenvalue $\lambda_1(r)= \lambda_1(B^N(p,r))$ is a decreasing function of $r$  and the limit at $r \to \infty$ does not depend on the choice of center for the balls $p$. This is a consequence of the Domain Monotonicity Principle. The exhaustion of $N$ by geodesic balls $\{B^N(p,r)\}_{r \to \infty}$ thence produces the {\em Dirichlet fundamental tone} of the manifold $N$:

\begin{equation*}
\lambda^*(N)=\lim_{r\to \infty} \lambda_1(B^N(p,r)) \quad .
\end{equation*}

The important challenge of finding upper and lower bounds of the first Dirichlet eigenvalue $\lambda_1(D)$ of a (small or large) normal domain $D \subset N$ in a given complete Riemannian manifold $N$ has received much attention by many authors since the classical works of Polya and Szeg\"{o} \cite{Po, PS}.\\

Concerning upper estimates, the classical and well-known comparison theorem by S. Y. Cheng, \cite{Cheng1}, gives a sharp upper bound for the first eigenvalue in terms of the diameter for Riemannian manifolds with Ricci curvatures bounded from below.

Cheng also proved a lower bound in \cite{Cheng2} via a comparison theorem for the first eigenvalue of a geodesic ball in a complete Riemannian manifold with sectional curvatures bounded from above by a given constant. If this bounding constant is negative, then we also have the lower bound for the fundamental tone of a Cartan-Hadamard manifold established by H. P. McKean in \cite{McK}.

McKean's inequality was later generalized by L. F. Cheung and P. F. Leung in \cite{CheLe} to the setting of complete non-compact submanifolds with bounded mean curvatures in hyperbolic spaces with constant negative curvature.

In this paper we present an exact expression of the first Dirichlet eigenvalue of general model space geodesic balls in terms of the so-called mean exit time moment spectrum for these balls. This is obtained by applying a fundamental Green operator bootstrapping technique due to S. Sato, \cite{Sa}. Using these model space results as comparison objects we are then also able to generalize the cited theorems by McKean and Cheung--Leung. The mean exit time moment spectrum has been considered by P. McDonald in \cite{Mc2, Mc} for a quite similar purpose. Since the moment spectrum thus plays a key r\^{o}le and obviously contains a wealth of geometric information (as does the Dirichlet spectrum) we will briefly introduce it already here:

Let us consider   the
induced Brownian motion $X_t$ defined on the Riemannian manifold $(N, g)$. We consider the functions $u_k$, $k\geq 0$, defined inductively as the following sequence of
solutions to a hierarchy of boundary value problems in $D \subset N$: First we let
\begin{equation}
u_{0} = 1\,\,\, \text{on}\,\,\, D \quad ,
\end{equation}
and then for $k\geq 1$ we define
\begin{equation}  \label{eqmoments1}
\begin{aligned}
\Delta u_k+k\, u_{k-1} &= 0\,\,\, \text{on}\,\,\, D \quad , \\
u_k\vert_{\partial D} &=0 \quad .
\end{aligned}
\end{equation}

The first non-trivial function in this sequence, $u_1(x)$, is the mean time of first
exit from $D$ for the Brownian motion starting at the point $x$ in
$D$, see \cite{Dy, Ma1}.

The $L^{1}$-moments of the exit time of $X_t$ from the smooth
precompact domain $D$, also called the {\em exit time moment
spectrum} of  $D$, $\{ \mathcal{A}_{k}(D)\}_{k=0}^{\infty}$, is
then given by the following integrals, see  \cite{Mc, Dy}:
\begin{equation}\label{defmoment}
\mathcal{A}_{k}(D)=  \int_D u_k(x)\, dV \quad .
\end{equation}

In particular, the quantity $\mathcal{A}_{1}(D)$ is known as the \emph{torsional rigidity}
of $D$. This name stems from the fact that  if $D \subset
\erre^2$, then $\mathcal{A}_{1}(D)$ represents the torque required
per unit angle of twist and per unit beam length
    when twisting an elastic beam of
    uniform cross section $D$, see \cite{Ba} and  \cite{PS}.

In the spirit of several previous seminal works, see e.g. \cite{Cha1, Cha2} and  \cite{BBC, BG}, we investigate to what extent the $L^{1}$-moment spectrum associated to the domain $D \subset N$ can replace, support, or estimate its corresponding Dirichlet spectrum in order to establish good descriptors for the geometry of $D$ and, eventually, for the geometry of the manifold $N$.

The first direct eigenvalue comparison findings in this direction are due to P. McDonald and R. Meyers, \cite{Mc2, Mc, McM}. As applied in \cite{McM}, we shall also use the following observation as another benchmark strategy for the results reported in the present paper, see Subsection \ref{subsec2}.
\begin{theorem} \label{thmMcDonaldExpress}
The first Dirichlet eigenvalue $\lambda_{1}(D)$ of any smooth domain $D \subset N^n$ can be directly extracted from the corresponding exit time moment spectrum
$\{ \mathcal{A}_{k}(D) \}_{k=k_{0}}^{\infty}$ as follows:
\begin{equation}\label{eqMcMexpress}
\lambda_1(D)= \sup \{\eta \geq 0 \,:\, \lim_{n \to \infty}  \sup \left(\frac{\eta}{2}\right)^n\frac{\mathcal{A}_{n}(D)}{\Gamma(n+1)} <\infty\} \quad .
\end{equation}
\end{theorem}

Here we restrict our studies to be concerned with the exit time moment spectra and the first Dirichlet eigenvalues of a specific kind of domains: Firstly the geodesic $R$-balls in rotationally symmetric (warped product) model spaces $M^{m}_{w}$ and secondly the so-called extrinsic $R$-balls $D_R \subset P^m$
of properly immersed submanifolds $P^m$ in ambient Riemannian manifolds $N^{n}$ with controlled sectional curvatures and possessing at least one pole $p \in N^n$, see \cite{S}.

\subsection{A first glimpse of the main results}

In this paper, we use a fundamental bootstrapping argument of S. Sato \cite{Sa} to obtain a more explicit expression than (\ref{eqMcMexpress}) for the first eigenvalue of the geodesic balls in the rotationally symmetric model spaces in terms of their moment spectra. Using the comparison techniques and results developed in \cite{HMP} we then obtain upper and lower estimates for the first Dirichlet eigenvalue of the extrinsic balls of a submanifold in a more general setting and, as a corollary we recover McKean's result and obtain also upper and lower bounds for the fundamental tone of a Cartan-Hadamard manifold and  the fundamental tone of a submanifold with controlled mean curvature along the lines of Cheung--Leung in \cite{CheLe} and Bessa--Montenegro in \cite{BM}.

The first of our main results is the following precise expression of the first Dirichlet eigenvalue of a geodesic ball in a rotationally symmetric model space $M^m_w$, which is defined as the warped product manifold $[\,0,\, \infty[ \times_{w} S^{m-1}_{1}$ with a pole $p$ (see Definition \ref{model} below). This expression implies an estimate as exact as you want for the first Dirichlet eigenvalue of a geodesic ball in the rotationally symmetric spaces, including the real space forms of constant curvature.

\begin{theoremA} \label{lambda1model}Let $B^w_{R}(p)$ be the geodesic ball of radius $R$ centered at the pole $p$ in
$M^m_w$. Then the first eigenvalue of the ball can be expressed as the following limits of exit time moment data:

\begin{displaymath}
\lambda_1(B^w_R)=\lim_{k\rightarrow
\infty}\frac{k\,{u}_{k-1}(0)}{{u}_k(0)}=\lim_{k\rightarrow
\infty}\frac{k\,\mathcal{A}_{k-1}(B^w_R)}{\mathcal{A}_{k}(B^w_R)},
\end{displaymath}
where ${u}_k$ are the functions defined by \eqref{eqmoments1}
and $\mathcal{A}_{k}(B^w_R)$ is the corresponding $k$-moment of $B^w_R$.
Moreover, the radial function
$g_\infty(r):=\lim_{k\rightarrow \infty}
\frac{{u}_k(r)}{{u}_k(0)}$ is a Dirichlet eigenfunction for the
first eigenvalue $\lambda_1(B^w_R)$.
\end{theoremA}

Using this theorem, we prove in Subsection \ref{subsec1} upper and lower bounds of the first Dirichlet eigenvalue of extrinsic balls in a submanifold together with intrinsic versions of these results and under a more relaxed set of curvature conditions.

As already alluded to above we also show how the techniques developed in the papers \cite{McM}, \cite{HMP} and  \cite{HMP2} can be applied  to give alternative proofs for some of these comparison results.

In order to illustrate our use of the upper and lower bounds on the ambient space sectional curvatures in the more general settings we extract here some consequences of Theorems \ref{th_const_below} and \ref{th_const_above}, under some specific restrictive assumptions that are actually not needed for the general versions of the theorems. The first of these results is a lower bound for the first Dirichlet eigenvalue of extrinsic balls in a  submanifold $P^m$ with controlled mean curvature $H_P$ in a Cartan-Hadamard manifold $N^n$.

To describe properly this control on the mean curvature, we need the following definition:

\begin{definition} Let us consider a pole $p$ in the ambient Cartan-Hadamard manifold $N$. The $p$-radial mean curvature function for $P^m$ in $N^n$ is defined in terms of the inner product of $H_{P}$ with the $N$-gradient of the distance function $r(x)$ from the pole $p$  as follows:
$$
\mathcal{C}(x) = -\langle \nabla^N r(x), H_{P}(x) \rangle  \quad
{\textrm{for all}}\quad x \in P \,\, .
$$
\end{definition}

Using a suitable radial control on the function  $\mathcal{C}(x)$ we then obtain:

\begin{theoremA}\label{cartanhad}
Let $N^n$ be a Cartan-Hadamard manifold, with sectional curvatures bounded from above by a constant $K_N \leq b \leq 0$. Let $p \in N$ be a pole in $N$. Let $P^m \subseteq N^n$ be a complete and non-compact properly immersed submanifold with $p$-radial mean curvature function $\mathcal{C}(x) \leq h(r(x))\,\,\textrm{for all} \,\, x \in P^m$, where $h(r)$ is a radial smooth function, called a {\em radial bounding function from above}.
Suppose that
\begin{equation}
(m-1)\cdot\sqrt{-b}\coth(R\sqrt{-b})\geq  m \cdot \sup_{r \in [0,R]}h(r) \quad ,
\end{equation}
where we read $\sqrt{-b}\coth(R\sqrt{-b})$ to be $1/R$ when $b=0$.

For any given  extrinsic ball $D_R(p)$ in $P^m$ we then have the following inequality:
\begin{equation} \label{eqLowerBound}
\lambda_1(D_R) \geq \frac{1}{4}\left((m-1)\cdot \sqrt{-b}\coth(R\sqrt{-b}) \,\, - m \cdot \sup_{r \in [0,R]}h(r)\right)^2 \quad .
\end{equation}

\end{theoremA}

The radial bounding function $h(r)$ is related to the global
extrinsic geometry of the submanifold. For example, it is obvious
that  minimal submanifolds $P^m$ satisfy that $\mathcal{C}(x)=0\,\,\textrm{for all}\,\, x \in P^m$, with bounding function $h=0$.
On the other hand, it can be proved, see the works
\cite{Sp,DCW,Pa1,MP5}, that when the submanifold is a convex
hypersurface, then we have $\mathcal{C}(x) \geq 0\,\, \textrm{for all}\,\, x \in P^m$
so the constant function $h(r)=0$ is thence
a radial bounding function \emph{from below}.

We obtain McKean's result from these theorems concerning the intrinsic situation, assuming that $P^m=N^n$. In this case the extrinsic domains $D_R$ become the geodesic balls $B^N_R$ of $N^n$ as we have pointed out above and we have moreover that $H_P(x)=0\,\textrm{for all} \,\, x \in P^m$, so we consider the bound $h(r)=0\, \textrm{for all}\,\, r>0$. Hence we obtain for $b < 0$ and for all $R>0$:
\begin{equation*}
\lambda_1(B^N_R) \geq \frac{1}{4}\left( (n-1)\,\sqrt{-b}\coth(\sqrt{-b}R)\right)^2 \quad .
\end{equation*}

For $b\leq 0\,\,\text{(we read $\sqrt{-b}\coth(R\sqrt{-b})$ to be $1/R$ when $b=0$)}$ we therefore have
\begin{equation}
\lambda^*(N)=\lim_{R\to \infty} \lambda_1(B^N_R)\geq \frac{(n-1)^{2}|b|}{4} \quad .
\end{equation}

We note that G. P. Bessa and J. F. Montenegro observe in \cite{BM}
an improvement of the bound (\ref{eqLowerBound}) in the intrinsic
setting as follows:

\begin{theoremA}[Bessa and Montenegro]

Under the intrinsic conditions with $N^{n}$ having sectional curvatures bounded from above by $b \leq 0$:
\begin{equation}
\lambda_1(B^N_R) \geq \frac{1}{4}\left(\max \left(\frac{n}{R} \, \, , \, \, (n-1)\cdot \sqrt{-b}\coth(R\sqrt{-b})\right) \right)^2 \quad .
\end{equation}
\end{theoremA}

As a consequence of Theorem \ref{cartanhad}, we also have a  generalization of Theorem 2 in Cheung--Leung's paper \cite{CheLe} and Corollary 4.4 in \cite{BM}. The ambient manifold may in our case be a general Cartan-Hadamard
manifold $N$ with sectional curvatures bounded from above by a constant $b < 0$, (as in \cite{BM}, while \cite{CheLe} considers the hyperbolic space $\hh^n(-1)$), and our submanifold has radial mean curvature bounded by a
radial function $h(r)$. This hypothesis includes complete and non-compact submanifolds $P^m$  in $N^{n}$ with mean curvature $H_P$ satisfying $\Vert H_P\Vert \leq \frac{\alpha}{m}$ with $\alpha \leq (m-1)\sqrt{-b}$, as
in Cheung-Leung's and Bessa-Montenegro's statements, which are then generalized as follows:
\begin{equation}
\lambda^*(P^m)=\lim_{R\to \infty} \lambda_1(B_R)\geq \frac{1}{4}\left((m-1)\sqrt{-b} - \alpha\right)^{2} \quad.
\end{equation}

As an application of the proofs and techniques developed here we also obtain, when we consider a Riemannian manifold with sectional curvatures bounded from above by the corresponding
sectional curvatures of a rotationally symmetric model space $M^n_w$, a direct first eigenvalue comparison result as follows:

\begin{theoremA}\label{chintrinsicabove}
Let $B^N_R(p)$ be a geodesic ball of a complete Riemannian manifold
$N^n$ with a pole $p$ and suppose that the $p$-radial sectional
curvatures of $N^n$ are bounded from above by the $p_w$-radial
sectional curvatures of a $w$-model space $M^n_w$.  Then
\begin{equation}\label{ineqleq_intrinsic}
 \lambda_1(B^N_R) \geq \lambda_1 (B^w_{R}),
 \end{equation}
where $B^{w}_{R}$ is the geodesic ball in $M^n_w$.
\end{theoremA}

As already mentioned we observe in Section \ref{comparisons}, that this result can also be obtained via the description given in \cite{McM} of the first Dirichlet eigenvalue of a domain $D$ in a
Riemannian manifold in terms of its exit time moment spectrum, i.e. Theorem \ref{thmMcDonaldExpress}, in combination with the isoperimetric type inequalities for the exit time moment spectra established in \cite{HMP2}. In the
 works \cite{Mc, Mc2} McDonald considers a complete Riemannian manifold $N$ which satisfies the following {\em moment comparison condition with respect to a constant curvature space form $\kan$}: For every smooth bounded precompact
 domain $D \subset N$ and for all $ k \in \ene$, assume that $\mathcal{A}_k(D) \leq \mathcal{A}_k(B^{b,n}_{T})$, where $B^{b,n}_{T}$ is a geodesic ball in $\kan$ such that $\Vol(D)=\Vol(B^{b,n}_T)$. Then
 $\lambda_1(D) \geq \lambda_1(B^{b,n}_T)$.

Using the same two-ways strategies  we also obtain  \emph{upper bounds} for the first Dirichlet eigenvalue of intrinsic balls as follows:

\begin{theoremA}\label{chintrinsicbelow}
Let $B^N_R$ be a geodesic ball of a complete Riemannian manifold
$N^n$ with a pole $p$ and suppose that the $p$-radial sectional
curvatures of $N^n$ are bounded from below by the $p_w$-radial
sectional curvatures of a $w$-model space $M^n_w$.  Then
\begin{equation}\label{ineqleq_intrinsic}
 \lambda_1(B^N_R) \leq \lambda_1 (B^w_{R}),
 \end{equation}
where $B^{w}_{R}$ is the geodesic ball in $M^n_w$.
\end{theoremA}

We remark that we also consider and prove extrinsic
generalizations of this result using somewhat more elaborate
comparison constellations which will be defined in Subsection
\ref{subsecCompConstel}, see Theorems    \ref{th_const_below} and
\ref{th_const_below2} in Subsections \ref{subsec1} and
\ref{subsec2}.  Moreover, in the intrinsic context given in
Theorems D and E, the equality with the bound is characterized
under some specific condition satisfied by the model space $M^n_w$
which serve as a curvature-controller, (see Theorems
\ref{th_const_below_intrinsic} and \ref{th_const_above_intrinsic}
in Subsection \ref{subsec3}).

\subsection{Outline of the paper}
Section \ref{pre} below is devoted to present all the preliminary concepts and instrumental prerequisites that we need in the following Sections \ref{secEigenMoment} and  \ref{comparisons} and which has
not been introduced already. Theorem A is proved in Section \ref{secEigenMoment}. The general and complete versions of Theorems B, D, and E are then presented and proved in Section \ref{comparisons}.  \\

\section{Preliminaries and Comparison Settings}\label{pre}

We first consider a few conditions and concepts that will be
instrumental for establishing our results in a more general version than presented in the introduction.

\subsection{The extrinsic balls and the curvature bounds} \label{subsecurvature}

We consider a properly immersed $m$-dimensional submanifold $P^m$
in a complete Riemannian manifold $N^n$. Let $p$ denote a point in
$P^m$ and assume that $p$ is a pole of the ambient manifold $N$. We
denote the distance function from $p$ in $N^{n}$ by $r(x) =
\dist_{N}(p, x)$ for all $x \in N$. Since $p$ is a pole there is -
by definition - a unique geodesic from $x$ to $p$ which realizes
the distance $r(x)$. We also denote by $r$ the restriction
$r\vert_P: P^m\longrightarrow \erre_{+} \cup \{0\}$. This
restriction is then called the extrinsic distance function from
$p$ in $P^m$. The corresponding extrinsic metric balls of
(sufficiently large) radius $R$ and center $p$ are denoted by
$D_R(p) \subset P^m$ and defined as any connected component which
contains $p$ of the set:
$$D_{R}(p) = B_{R}(p) \cap P =\{x\in P \,|\, r(x)< R\} \quad ,$$
where $B_{R}(p)$ denotes the geodesic $R$-ball around the pole $p$
in $N^n$. The
 extrinsic ball $D_R(p)$ is a connected domain in $P^m$, with
boundary $\partial D_{R}(p)$. Since $P^{m}$ is assumed to be
unbounded in $N$ we have for every sufficiently large $R$ that
$B_{R}(p) \cap P \neq P$.
\bigskip

We now present the curvature restrictions which constitute the
geometric framework of our investigations.

\begin{definition}
Let $p$ be a point in a Riemannian manifold $M$ and let $x \in
M-\{ p \}$. The sectional curvature $K_{M}(\sigma_{x})$ of the
two-plane $\sigma_{x} \in T_{x}M$ is then called a
\textit{$p$-radial sectional curvature} of $M$ at $x$ if
$\sigma_{x}$ contains the tangent vector to a minimal geodesic
from $p$ to $x$. We denote these curvatures by $K_{p,
M}(\sigma_{x})$.
\end{definition}

Another notion needed to describe our comparison setting is the
idea of {\it radial tangency}. If we denote by $\nabla r$ and
$\nabla^P r$ the
 gradients of
$r$ in $N$ and $P$ respectively, then we have the following basic
relation:
\begin{equation}\label{eq2.1}
\nabla r = \nabla^P r +(\nabla r)^\bot \quad ,
\end{equation}
where $(\nabla r)^\bot(q)$ is perpendicular to $T_qP$ for all $q\in P$.\\

When the submanifold $P$ is totally geodesic, then $\nabla
r=\nabla^P r$ in all points, and, hence, $\Vert \nabla^P r\Vert
=1$. On the other hand, and given the starting point $p \in P$,
from which we are measuring the distance $r$, we know that $\nabla
r(p)=\nabla^P r(p)$, so $\Vert \nabla^P r(p)\Vert =1$. Therefore,
the difference  $1 - \Vert \nabla^P r\Vert$ quantifies the radial
{\em detour} of the submanifold with respect the ambient manifold
as seen from the pole $p$. To control this detour locally, we
apply the following

\begin{definition}

 We say that the submanifold $P$ satisfies a {\it radial tangency
 condition} at $p\in P$ when we have a smooth positive function
 $$g: P \mapsto \erre_{+} \,\, ,$$ so that
\begin{equation}
\mathcal{T}(x) \, = \, \Vert \nabla^P r(x)\Vert \geq g(r(x)) \, >
\, 0  \quad {\textrm{for all}} \quad x \in P \,\, .
\end{equation}
\end{definition}
\begin{remark}

Of course, we always have

\begin{equation}
\mathcal{T}(x) \, =\, \Vert \nabla^P r(x)\Vert \leq1 \quad
{\textrm{for all}} \quad x \in P \,\, .
\end{equation}
\end{remark}

\subsection{Model Spaces} \label{secModel}

As mentioned previously, the model spaces $M^m_w$ serve foremost
as com\-pa\-ri\-son controllers for the radial sectional
curvatures of $N^{n}$.

\begin{definition}\label{model}[See \cite{Gri}, \cite{GreW}, \cite{Pe}]
 A $w-$model $M_{w}^{m}$ is a
smooth warped product with base $B^{1} = [\,0,\, R[ \,\,\subset\,
\mathbb{R}$ (where $\, 0 < R \leq \infty$\,), fiber $F^{m-1} =
S^{m-1}_{1}$ (i.e. the unit $(m-1)-$sphere with standard metric),
and warping function $w:\, [\,0, \,R[\, \to \mathbb{R}_{+}\cup
\{0\}\,$ with $w(0) = 0$, $w'(0) = 1$, $w^{(k)}(0) = 0$ for all even derivation orders $k$ and $w(r) > 0\,$ for all
$\, r > 0\,$. The point $p_{w} = \pi^{-1}(0)$, where $\pi$ denotes
the projection onto $B^1$, is called the {\em{center point}} of
the model space. If $R = \infty$, then $p_{w}$ is a pole of
$M_{w}^{m}$.
\end{definition}

\begin{remark}\label{propSpaceForm}
The simply connected space forms $\kam$ of constant curvature $b$
can be constructed as  $w-$models with any given point as center
point using the warping functions
\begin{equation}
w(r) = w_{b}(r) =\begin{cases} \frac{1}{\sqrt{b}}\sin(\sqrt{b}\, r) &\text{if $b>0$}\\
\phantom{\frac{1}{\sqrt{b}}} r &\text{if $b=0$}\\
\frac{1}{\sqrt{-b}}\sinh(\sqrt{-b}\,r) &\text{if $b<0$} \quad .
\end{cases}
\end{equation}
Note that for $b > 0$ the function $w_{b}(r)$ admits a smooth
extension to  $r = \pi/\sqrt{b}$. For $\, b \leq 0\,$ any center
point is a pole.
\end{remark}

In the papers \cite{O'N,GreW,Gri,MP3,MP4}, we have a complete
description of these model spaces, including the computation of
their sectional curvatures $K_{p_{w} , M_{w}}$ in the radial
directions from the center point. They are determined by the
radial function $K_{p_{w} , M_{w}}(\sigma_{x}) \, = \, K_{w}(r) \,
= \, -\frac{w''(r)}{w(r)}$. Moreover,
 the mean curvature of the distance sphere of radius $r$ from the center point is
\begin{equation}\label{eqWarpMean}
\eta_{w}(r)  = \frac{w'(r)}{w(r)} = \frac{d}{dr}\ln(w(r))\quad .
\end{equation}

\subsection{The Isoperimetric Comparison Space}

\label{secIsopCompSpace} Given the bounding functions $g(r)$,
$h(r)$ and the ambient curvature controller function $w(r)$
described is Subsections \ref{subsecurvature} and \ref{secModel},
as in \cite{MP5,HMP} we construct a new model space $C^{\,m}_{w,
g, h}\,$. For completeness, we recall this construction:

\begin{definition}\label{stretchingfunct}
Given a smooth positive function
$$g: P \mapsto \erre_{+} \,\, ,$$
satisfying $g(0)=1$ and $g(r(x))\leq 1\,\,{\textrm{for all \,}} x
\in P$, a 'stretching' function $s$ is defined as follows
\begin{equation}\label{eqstretching}
s(r) \, = \, \int_{0}^{r}\,\frac{1}{g(t)} \, dt \quad .
\end{equation}
It  has a well-defined inverse $r(s)$ for $s \in [\,0, s(R)\,]$
with derivative $r'(s) \, = \, g(r(s))$. In particular $r'(0)\, =
\, g(0) \, = \, 1$.
\end{definition}

\begin{definition}[\cite{MP5}] \label{defCspace}
The {\em{isoperimetric comparison space}} $C^{\,m}_{w, g, h}\,$ is
the $W-$model space with base interval $B\,= \, [\,0, s(R)\,]$ and
warping function $W(s)$ defined by
\begin{equation}\label{defW}
W(s) \, = \, \Lambda^{\frac{1}{m-1}}(r(s)) \quad ,
\end{equation}
where the auxiliary function $\Lambda(r)$ satisfies the following
differential equation:
\begin{equation} \label{eqLambdaDiffeq}
\begin{aligned}
\frac{d}{dr}\,\{\Lambda(r)w(r)g(r)\} \, &= \, \Lambda(r)w(r)g(r)\left(\frac{m}{g^{2}(r)}\left(\eta_{w}(r) - h(r) \right)\right) \\
&= \, m\,\frac{\Lambda(r)}{g(r)}\left(w'(r) - h(r)w(r)
\right)\quad,
\end{aligned}
\end{equation}
and the following boundary condition:
\begin{equation} \label{eqTR}
\frac{d}{dr}_{|_{r=0}}\left(\Lambda^{\frac{1}{m-1}}(r)\right) = 1
\quad .
\end{equation}

\end {definition}

We observe, that in spite of its relatively complicated
construction, $C^{\,m}_{w, g, h}\,$ is indeed a model space
$M^m_W$ with a well defined pole $p_{W}$ at $s = 0$: $W(s) \geq 0$
for all $s$ and $W(s)$ is only $0$ at $s=0$, where also, because
of the explicit construction in definition \ref{defCspace} and
because of
 equation (\ref{eqTR}):  $W'(0)\, =
\, 1\,$, and moreover, $W^{(k)}(0)=0$ for all even $k$, (see \cite{MP5}).\\


Note that, when $g(r)=1 \,\,\,{\textrm{for all \,}} r$ and
$h(r)=0\,\,\,{\textrm{for all \,\,}} r$, then the stretching
function $s(r)=r$ and $W(s(r))=w(r) \,\,\,{\textrm{for all \,}}
r$, so $C_{w,g,h}^{m}$ becomes
a model space with warping function $w$, $M^m_w$.\\

 These are the spaces where the bounds on the $L^1$-moment spectrum are attained. We shall refer
to the $W$-model spaces $M^{m}_{W} = C_{w,g,h}^{m}$ as the {\em
isoperimetric comparison spaces} of dimension $m$ defined by the radial functions $w$, $g$, and $h$.

\subsection{Balance conditions}

In the paper \cite{HMP} we imposed a balance condition on the
general model spaces $M^m_W$, that we will need in the sequel:

\begin{definition} \label{defBalCond}
The model space $M_{W}^{m} \, = \, C_{w, g, h}^{m}$ is
{\em{$w-$balanced from below}} (with respect to the intermediary
model space $M_{w}^{m}$) if the following holds for all $r \in \,
[\,0, R\,]$, resp. all $s \in \, [\,0, s(R)\,]$:
\begin{equation}\label{eqBalA}
q_{W}(s)\left(\eta_{w}(r(s)) - h(r(s)) \right) \, \geq g(r(s))/m
\quad .
\end{equation}
Here $q_{W}(s)$ is the isoperimetric quotient function
\begin{equation} \label{eqIsopFunction}
\begin{aligned}
q_{W}(s) \, &= \, \frac{\Vol(B_{s}^{W})}{\Vol(S_{s}^{W})} \,
\\ &= \, \frac{\int_{0}^{s}\,W^{m-1}(t)\,dt}{W^{m-1}(s)}\,
\\ &= \,
\frac{\int_{0}^{r(s)}\,\frac{\Lambda(u)}{g(u)}\,du}{\Lambda(r(s))}
\quad .
\end{aligned}
\end{equation}

\end{definition}

\begin{remark}
 In particular the $w$-balance condition from below  for $M_{W}^{m} \, = \, C_{w, g, h}^{m}$ implies that
\begin{equation} \label{eqEtaVSh}
\eta_{w}(r) \, - h(r) \, > \, 0 \quad .
\end{equation}
\end{remark}

\begin{remark}
The  definition of $w-$balance condition from below for
$M_{W}^{m}$ is clearly an extension of the balance condition from
below as defined in \cite[Definition 2.12]{MP4}. The condition in
that paper is obtained precisely when $g(r) \, = \, 1$ and $h(r)
\, = \, 0$ for all $r \in [\,0, R]\,$ so that $r(s) \, =\, s$,
$W(s)\, = \, w(r)$, and
\begin{equation}
q_{w}(r)\eta_{w}(r)\, \geq 1/m \quad .
\end{equation}
In particular, the Hyperbolic spaces $\hh^n(b)$ are balanced from below.
\end{remark}

\subsection{Comparison Constellations} \label{subsecCompConstel}

We now present the precise settings where our main results take
place, introducing the notion of {\em comparison constellations}
as it was previously defined in \cite{HMP}. For that purpose we
shall bound the previously introduced notions of radial curvature
and tangency by the corresponding quantities attained in some
special model spaces, called {\em isoperimetric comparison spaces}
to be defined in the next subsection.

\begin{definition}\label{defConstellatNew1}
Let $N^{n}$ denote a complete Riemannian manifold with a pole $p$
and distance function $r \, = \, r(x) \, = \, \dist_{N}(p, x)$.
Let $P^{m}$ denote an unbounded complete and closed submanifold in
$N^{n}$. Suppose $p \in P^m$ and suppose that the following
conditions are satisfied for all $x \in P^{m}$ with $r(x) \in
[\,0, R]\,$:
\begin{enumerate}
\item The $p$-radial sectional curvatures of $N$ are bounded from
below by the $p_{w}$-radial sectional curvatures of of the
$w-$model space $M_{w}^{m}$:
$$
\mathcal{K}(\sigma_{x}) \, \geq \, -\frac{w''(r(x))}{w(r(x))}
\quad .
$$

\item The $p$-radial mean curvature of $P$ is bounded from below
by a smooth radial function $h(r)$:
$$
\mathcal{C}(x)  \geq h(r(x)) \quad.
$$

\item The submanifold $P$ satisfies a {\it radial tangency
 condition} at $p\in P$, with smooth positive function $g$ i.e. we have a smooth positive function
 $$g: P \mapsto \erre_{+} \,\, ,$$ such that
\begin{equation}
\mathcal{T}(x) \, = \, \Vert \nabla^P r(x)\Vert \geq g(r(x)) \, >
\, 0  \quad {\textrm{for all}} \quad x \in P \,\, .
\end{equation}
\end{enumerate}
Let $C_{w,g,h}^{m}$ denote the $W$-model with the specific warping
function $W: \pi(C_{w,g,h}^{m}) \to \mathbb{R}_{+}$  constructed
in Definition \ref{defCspace}, (Subsection
\ref{secIsopCompSpace}), via $w$, $g$, and $h$. Then the triple
$\{ N^{n}, P^{m}, C_{w,g,h}^{m} \}$ is called an
{\em{isoperimetric comparison constellation bounded from below}}
on the interval $[\,0, R]\,$.
\end{definition}

A \lq\lq constellation bounded from above" is given by the
following dual setting, (with respect to the definition above),
considering the special $W$-model spaces $C_{w,g,h}^{m}$ with
$g=1$:

\begin{definition}\label{defConstellatNew2}
Let $N^{n}$ denote a Riemannian manifold with a pole $p$ and
distance function $r \, = \, r(x) \, = \, \dist_{N}(p, x)$. Let
$P^{m}$ denote an unbounded complete and closed submanifold in
$N^{n}$. Suppose the following  conditions are satisfied for all
$x \in P^{m}$ with $r(x) \in [\,0, R]\,$:
\begin{enumerate}
\item The $p$-radial sectional curvatures of $N$ are bounded from
above by the $p_{w}$-radial sectional curvatures of the $w-$model
space $M_{w}^{m}$:
$$
\mathcal{K}(\sigma_{x}) \, \leq \, -\frac{w''(r(x))}{w(r(x))}
\quad .
$$

\item The $p$-radial mean curvature of $P$ is bounded from above
by a smooth radial function $h(r)$:
$$
\mathcal{C}(x)  \leq h(r(x)) \quad.
$$
\end{enumerate}

Let $C_{w,1,h}^{m}$ denote the $W$-model with the specific warping
function $W: \pi(C_{w,1,h}^{m}) \to \mathbb{R}_{+}$ constructed,
(in the same way as in Definition \ref{defConstellatNew1} above),
in Definition \ref{defCspace} via $w$, $g=1$, and $h$. Then the
triple $\{ N^{n}, P^{m}, C_{w,1,h}^{m} \}$ is called an
{\em{isoperimetric comparison constellation bounded from above}}
on the interval $[\,0, R]\,$.
\end{definition}

\subsection{Laplacian Comparison}
\label{secLaplacecompar}

 We begin this section recalling the following
Laplacian comparison Theorem for manifolds with a pole (see
\cite{GreW,JK,MP3,MP4,MP5,MM} for more details).
\begin{theorem} \label{corLapComp} Let $N^{n}$ be a manifold with a pole $p$, let $M_{w}^{m}$ denote a
$w-$model space with center $p_{w}$. Then we have the following dual Laplacian inequalities for modified distance functions:\\

(i) Suppose that every $p$-radial sectional curvature at $x \in N
- \{p\}$ is bounded  by the $p_{w}$-radial sectional curvatures in
$M_{w}^{m}$ as follows:
\begin{equation}
\mathcal{K}(\sigma(x)) \, = \, K_{p, N}(\sigma_{x})
\geq-\frac{w''(r)}{w(r)}\quad .
\end{equation}

Then we have for every smooth function $f(r)$ with $f'(r) \leq
0\,\,\textrm{for all}\,\,\, r$, (respectively $f'(r) \geq
0\,\,\textrm{for all}\,\,\, r$):
\begin{equation} \label{eqLap1}
\begin{aligned}
\Delta^{P}(f \circ r) \, \geq (\leq) \, &\left(\, f''(r) -
f'(r)\eta_{w}(r) \, \right)
 \Vert \nabla^{P} r \Vert^{2} \\ &+ mf'(r) \left(\, \eta_{w}(r) +
\langle \, \nabla^{N}r, \, H_{P}  \, \rangle  \, \right)  \quad ,
\end{aligned}
\end{equation}
where $H_{P}$ denotes the mean curvature vector
of $P$ in $N$.\\

(ii) Suppose that every $p$-radial sectional curvature at $x \in N
- \{p\}$ is bounded  by the $p_{w}$-radial sectional curvatures in
$M_{w}^{m}$ as follows:
\begin{equation}
\mathcal{K}(\sigma(x)) \, = \, K_{p, N}(\sigma_{x})
\leq-\frac{w''(r)}{w(r)}\quad .
\end{equation}

Then we have for every smooth function $f(r)$ with $f'(r) \leq
0\,\,\textrm{for all}\,\,\, r$, (respectively $f'(r) \geq
0\,\,\textrm{for all}\,\,\, r$):
\begin{equation} \label{eqLap2}
\begin{aligned}
\Delta^{P}(f \circ r) \, \leq (\geq) \, &\left(\, f''(r) -
f'(r)\eta_{w}(r) \, \right)
 \Vert \nabla^{P} r \Vert^{2} \\ &+ mf'(r) \left(\, \eta_{w}(r) +
\langle \, \nabla^{N}r, \, H_{P}  \, \rangle  \, \right)  \quad ,
\end{aligned}
\end{equation}
where $H_{P}$ denotes the mean curvature vector of $P$ in $N$.
\end{theorem}

\section{The first eigenvalue and moment spectra \\ of geodesic balls in model spaces} \label{secEigenMoment}
The aim of this section is to obtain the first eigenvalue of a
geodesic ball $B^w_R$, namely $\lambda_1 (B^w_R)$, in terms of the
$L^1$-moment spectrum of  $B^w_R$.

To do that, we divide this section in two subsections. In the first, we shall provide a precise integral description of the $L^1$-moment spectrum of  $B^w_R$ as previously also presented in \cite{HMP}). In the second we shall relate the first Dirichlet eigenvalue of the geodesic balls with its moment spectrum using a bootstrapping technique for the Green operator due to S. Sato, \cite{Sa}.
\subsection{The moment
spectrum of a geodesic model $R$-ball} \label{SpectrumBall}

We have the following result concerning the $L^1$-moment spectrum
of a geodesic $R$-ball $B^w_R \subset M^m_w$, see \cite{HMP} for its proofs:

\begin{proposition}\label{propW1} Let $\tilde{u}_k$ be the solution of the  boundary value problems \eqref{eqmoments1},
defined on the geodesic $R$-ball $B^w_R$ in a warped model space
$M^m_w$.

Then
\begin{equation}\label{eq_ukint}
\tilde{u}_k(r)=k\,\int_r^R \frac{\int_0^t w^{m-1}(s)
\tilde{u}_{k-1}(s)\, ds}{w^{m-1}(t)}\,dt,
\end{equation}
and
\begin{equation}\label{eq_uk}
\tilde{u}_k'(r) \, = -k\,\frac{\int_0^r w^{m-1}(s)
\tilde{u}_{k-1}(s)\, ds}{w^{m-1}(r)}.
\end{equation}
Therefore, applying the Divergence Theorem
\begin{equation}\label{momentsW}
\mathcal{A}_k(B^w_R)=-\frac{1}{k+1}\, \tilde{u}_{k+1}'(R)
\,\Vol(S^w_R),
\end{equation}
where $S^w_R$ is the geodesic $R$-sphere in $M^m_w$ .
\end{proposition}

Let us consider now $M^m_W$ be an isoperimetric comparison model
space and let $\tilde{u}^W_k$ be the radial functions given by
\eqref{eq_uk}, which are the solutions of the problems
\eqref{eqmoments1} defined on the geodesic ball $B^W_{s(R)}$. We
define the functions $f_k: [\,0,R]\rightarrow \mathbb{R}$ as
$f_k=\tilde{u}^W_k\circ s$, where $s$ is the stretching function
given by \eqref{eqstretching}. Then, we have the following
\begin{lemma}\label{paren} Let $M^m_W$ be an  isoperimetric comparison model
space that is $w$-balanced from below in the sense of Definition
\ref{defBalCond}. Then for all $k\geq 1$,
\begin{displaymath}
f_k''(r)-f_k'(r)\eta_w(r)\geq 0.
\end{displaymath}
\end{lemma}

\subsection{The first Dirichlet eigenvalue of geodesic balls in the model spaces $M^m_w$}
Since $M^m_w$ is a spherically symmetric manifold with pole $p$, the Green operator for radial
functions on $B^w_R(p)$ is given by

\begin{equation} \label{eqGreen}
G(u)(r)=\int_r^R \frac{\int_0^t w^{m-1}(s) u(s)\,
ds}{w^{m-1}(t)}\,dt.
\end{equation}

Note that the moment functions $\tilde{u}_k$ given by
\eqref{eqmoments1}, can be written via the Green operator as
$\tilde{u}_k(r)=k\,G(\tilde{u}_{k-1})(r)$. Using these functions,
we can prove the following estimate
\begin{proposition}\label{inequalitylamba1}
Let $\tilde{u}_k$ be the radial functions defined on $D = B^w_R(p)$ and
given by \eqref{eqmoments1}. Then, for all $k\geq 1$, the functions $\frac{k\,
\tilde{u}_{k-1}}{\tilde{u}_{k}}(r)$ are increasing and we have

\begin{eqnarray}\label{lamba1_emparedado}
\frac{k\, \tilde{u}_{k-1}}{\tilde{u}_ k}(0)\leq
\lambda_1(B^w_R)\leq \frac{k\,
\mathcal{A}_{k-1}(B^w_R)}{\mathcal{A}_k(B^w_R)},
\end{eqnarray}
where $\mathcal{A}_k(B^w_R)$ is the $k$-moment of $B^w_R$.
\end{proposition}
\begin{proof}
As $\tilde{u}_k$ is a positive radial $C^2$-function on $B^w_R$
satisfying that $\tilde{u}_k(R)=0$, a direct application of
Barta's inequalities, (see \cite{Bar} and \cite{Cha1}) give us that

\begin{equation}\label{lambda_1Barta}
\inf_{B^w_R} \Big(\frac{k\,
\tilde{u}_{k-1}}{\tilde{u}_k}(r)\Big)\leq \lambda_1(B^w_R)\leq
\sup_{B^w_R}\Big(\frac{k\,
\tilde{u}_{k-1}}{\tilde{u}_k}(r)\Big),\qquad \textrm{for all }
k\geq 1.
\end{equation}

Now, we are going to show inductively that the functions $k\,
\tilde{u}_{k-1}(r)/\tilde{u}_{k}(r)$ are increasing for all $k
\geq 1$ and consequently the value of the infimum and the supremum
are attained in $r=0$ and $r=R$ respectively.

So, we begin studying the quotient $2
\tilde{u}_1(r)/\tilde{u}_2(r)$ (note that we can assume that
$\tilde{u}_0\equiv 1$, and $1/\tilde{u}_1(r)$ would be
increasing). By \eqref{eq_ukint} and \eqref{eq_uk}  the first
derivative of the quotient is given by

\begin{eqnarray*}
&&\big(\frac{2 \tilde{u}_1}{\tilde{u}_2}\big)'(r)=
4\frac{-G(\tilde{u}_1)(r)\,\int_0^r w^{m-1}(s)\,ds
+G(1)(r)\,\int_0^r w^{m-1}(s) \tilde{u}_1(s)\,ds}{w^{m-1}(r)
(\tilde{u}_2(r))^2}\\
&=& \frac{\int_r^R \frac{-4}{w^{m-1}(t)}\Big(\int_0^r
w^{m-1}(s)\,ds \int_0^t w^{m-1}(s) \tilde{u}_1(s)\,ds-\int_0^t
w^{m-1}(s)\,ds \int_0^r w^{m-1}(s) \tilde{u}_1(s)\,ds
\Big)\,dt}{w^{m-1}(r) (\tilde{u}_2(r))^2} \quad .
\end{eqnarray*}

We define $h(r):=\frac{\int_0^r w^{m-1}(s)\,ds}{\int_0^r
w^{m-1}(s) \tilde{u}_1(s)\,ds}$. Since $\tilde{u}_1$ is a
decreasing function,

\begin{eqnarray*}
h'(r)&=&\frac{w^{m-1}(r)\int_0^r w^{m-1} (s) \tilde{u}_1(s)\,ds -
w^{m-1}(r) \tilde{u}_1(r)\int_0^r w^{m-1}(s)\,ds}{(\int_0^r
w^{m-1}(s) \tilde{u}_1(s)\,ds)^2}\geq 0,
\end{eqnarray*}
and $h(r)\leq h(t)$ for all $t\geq r$. Therefore,

\begin{displaymath}
\int_0^r w^{m-1}(s)\,ds \int_0^t w^{m-1}(s)
\tilde{u}_1(s)\,ds-\int_0^t w^{m-1}(s)\,ds \int_0^r w^{m-1}(s)
\tilde{u}_1(s)\,ds \leq 0,
\end{displaymath}

and $\big(\frac{2 \tilde{u}_1}{\tilde{u}_2}\big)'(r)\geq 0$. Now,
if we assume that $k\, \tilde{u}_{k-1}/\tilde{u}_{k}=
\tilde{u}_{k-1}/G(\tilde{u}_{k-1})$ is increasing, applying Lemma
2 of \cite{Sa} we obtain that
$(k+1)\,\tilde{u}_k/\tilde{u}_{k+1}=G(\tilde{u}_{k-1})/G^2(\tilde{u}_{k-1})$
is also increasing and we are done.

By virtue of the monoticity of the quotients
$k\,\tilde{u}_{k-1}(r)/\tilde{u}_k(r)$ equation
\eqref{lambda_1Barta} reads
\begin{equation}\label{Barta2}
\inf_{B^w_R} \Big(\frac{k\,
\tilde{u}_{k-1}}{\tilde{u}_k}(r)\Big)=\frac{k\,
\tilde{u}_{k-1}}{\tilde{u}_k}(0)\leq \lambda_1(B^w_R)\leq
\sup_{B^w_R}\Big(\frac{k\, \tilde{u}_{k-1}}{\tilde{u}_k}(r)\Big)=
\frac{k\, \tilde{u}_{k-1}}{\tilde{u}_k}(R).
\end{equation}

To obtain \eqref{lamba1_emparedado} from inequalities
\eqref{Barta2}, note that applying L'Hopital and taking into
account equation \eqref{momentsW}, for $k \geq 2$,
\begin{equation}
\begin{aligned}
k\,\tilde{u}_{k-1}(R)/\tilde{u}_{k}(R)&=\lim_{r \to R}k\,\tilde{u}_{k-1}(r)/\tilde{u}_{k}(r)\\
&=\lim_{r \to
R}\frac{k\,\tilde{u}_{k-1}'(r)}{\tilde{u}_{k}'(r)}=\frac{(k-1)\,
\mathcal{A}_{k-2}(B^w_R)}{\mathcal{A}_{k-1}(B^w_R)} \quad .
\end{aligned}
\end{equation}
\end{proof}

When $k=1$, the above proposition gives us immediately the
following upper and lower bounds for the first eigenvalue of a
geodesic ball. The lower bound was obtained in \cite{BCG} for
geodesic balls in the $n$-dimensional sphere $\mathbb{S}^n(1)$,
and later it was generalized for an arbitrary $M^m_w$ in
\cite{BaBe}. The upper bound give us a new relation between the
first Dirichlet eigenvalue and the torsional rigidity of a
geodesic ball in a rotationally symmetric space.

\begin{corollary}
The first eigenvalue $\lambda_1(B^w_R)$ of the geodesic balls in the rotationally symmetric spaces $M^m_w$ satisfies
\begin{equation}\label{UpperLower}
\frac{1}{\int_0^R
q_w(t)\,dt}\leq \lambda_1(B^w_R)\leq  \frac{{\rm vol}(B^w_R)}{\mathcal{A}_1(B^w_R)} \quad ,
\end{equation}
where $q_w(t)$ is the isoperimetric quotient defined by
\begin{equation}\label{isoquotient}
q_w(t):=\frac{\int_0^t w^{m-1}(s)\, ds}{w^{m-1}(t)}.
\end{equation}
\end{corollary}

We are going now
to show that Proposition \ref{inequalitylamba1} give us better
estimations of the first Dirichlet eigenvalue of the geodesic balls in rotationally symmetric spaces, because  inequalities \eqref{lamba1_emparedado} improve when
$k$ increases.

\begin{corollary} Let $\tilde{u}_k$ be the functions defined on $B^w_R$ and given by
\eqref{eq_ukint}. Then,

\begin{equation}\label{ineq_limit}
\lim_{k\rightarrow \infty}
\frac{k\,\tilde{u}_{k-1}}{\tilde{u}_k}(0)\leq \lambda_1(B^w_R)\leq
\lim_{k\rightarrow\infty} \frac{k\,
\mathcal{A}_{k-1}(B^w_R)}{\mathcal{A}_{k}(B^w_R)}.
\end{equation}
\end{corollary}
\begin{proof}
Since $G$ is a positive operator and
$(k+1)\,\tilde{u}_k/\tilde{u}_{k+1}=G(\tilde{u}_{k-1})/G^2(\tilde{u}_{k-1})$,
applying Lemma 1 of \cite{Sa}, we obtain that for $k\geq 1$, and
for all $r \in [0,R]$:

\begin{equation}\label{eq1}
\inf_{B^w_R} \Big(\frac{k\,
\tilde{u}_{k-1}}{\tilde{u}_k}(r)\Big)=\frac{k\,
\tilde{u}_{k-1}}{\tilde{u}_k}(0)\leq \frac{(k+1)\,
\tilde{u}_{k}}{\tilde{u}_{k+1}}(r)\leq \sup_{B^w_R}\Big(\frac{k\,
\tilde{u}_{k-1}}{\tilde{u}_k}(r)\Big)= \frac{k\,
\tilde{u}_{k-1}}{\tilde{u}_k}(R).
\end{equation}

In particular, for $r=0$ and for all $k \geq 1$, we have
\begin{equation}\label{eq2}
\frac{k\, \tilde{u}_{k-1}}{\tilde{u}_k}(0)\leq \frac{(k+1)\,
\tilde{u}_{k}}{\tilde{u}_{k+1}}(0) \quad .
\end{equation}

Moreover, by Proposition \ref{inequalitylamba1}, we know that, for
all $k \geq 1$, $\frac{k\, \tilde{u}_{k-1}}{\tilde{u}_k}(0) \leq
\lambda_1(B^w_R)$, so therefore $\{\frac{k\,
\tilde{u}_{k-1}}{\tilde{u}_k}(0)\}_{k=1}^{\infty}$ is a bounded
increasing sequence, so there exists
\[\lambda:= \lim_{k\rightarrow \infty} \frac{k\,\tilde{u}_{k-1}}{\tilde{u}_k}(0) \leq \lambda_1(B^w_R).\]

On the other hand, taking $r=R$ in (\ref{eq1}) we obtain, for all $k \geq 1$

\begin{equation}\label{eq3}
\frac{(k+1)\, \tilde{u}_{k}}{\tilde{u}_{k+1}}(R)\leq\frac{k\,
\tilde{u}_{k-1}}{\tilde{u}_k}(R).
\end{equation}

Moreover, by inequality \eqref{Barta2} in Proposition
\ref{inequalitylamba1}, we know that, for all $k \geq 1$,
$\frac{k\, \tilde{u}_{k-1}}{\tilde{u}_k}(R) \geq
\lambda_1(B^w_R)$, so therefore $\{\frac{k\,
\tilde{u}_{k-1}}{\tilde{u}_k}(R)\}_{k=1}^{\infty}$  is a
decreasing sequence bounded from below with the limit
\[\mu:= \lim_{k\rightarrow \infty} \frac{k\,\tilde{u}_{k-1}}{\tilde{u}_k}(R)=\lim_{k\rightarrow \infty} \frac{k\,
\mathcal{A}_{k-1}(B^w_R)}{\mathcal{A}_{k}(B^w_R)}\geq
\lambda_1(B^w_R),\]

and this concludes the proof.
\end{proof}

We are now able to state and prove the main result of this section, which is also stated as Theorem A  in the introduction:

\begin{theorem} \label{lambda1model}Let $B^w_R$ be the geodesic ball of radius $R$ in
$M^m_w$. Then,

\begin{displaymath}
\lambda_1(B^w_R)=\lim_{k\rightarrow
\infty}\frac{k\,\tilde{u}_{k-1}(0)}{\tilde{u}_k(0)}=\lim_{k\rightarrow
\infty}\frac{k\,\mathcal{A}_{k-1}(B^w_R)}{\mathcal{A}_{k}(B^w_R)},
\end{displaymath}
where $\tilde{u}_k$ are the functions defined by \eqref{eq_ukint}
and $\mathcal{A}_{k}(B^w_R)$ is the $k$-moment of $B^w_R$.
Moreover, the radial $C^2$-function
$g_\infty(r):=\lim_{k\rightarrow \infty}
\frac{\tilde{u}_k(r)}{\tilde{u}_k(0)}$ is an eigenfunction of the
first eigenvalue.
\end{theorem}
\begin{proof}

We are going to apply Barta's Lemma to show equality
\begin{displaymath}
\lambda_1(B^w_R)=\lim_{k\rightarrow
\infty}\frac{k\,\tilde{u}_{k-1}(0)}{\tilde{u}_k(0)} \quad .
\end{displaymath}

For that, we shall first define (see Lemma \ref{C2limit} below) a positive radial $C^2$-function $g_\infty(r)$ defined on $B^w_R$ such that
$g_\infty(R)=0$ and such that satisfies

\begin{displaymath}
\Delta
g_\infty(r)=g_\infty''(r)+(m-1)\eta_w(r)g_\infty'(r)=-\lambda\,
g_\infty(r) \quad ,
\end{displaymath}
\noindent where $\lambda= \lim_{k\rightarrow \infty} \frac{k\,\tilde{u}_{k-1}}{\tilde{u}_k}(0)$.

Then, $\lambda$ is a Laplacian eigenvalue on $B^w_R$ with
eigenfunction $g_\infty(r)$. As $g_\infty(R)=0$, Barta's
inequalities imply that

\begin{equation}
\lambda=\inf_{B^w_R} \Big(\frac{-\Delta g_\infty}{g_\infty}\Big) \leq \lambda_1(B^w_R) \leq
\sup_{B^w_R}\Big(\frac{-\Delta g_\infty}{g_\infty}\Big)= \lambda \quad ,
\end{equation}

and hence
\begin{equation}
\lambda=\lim_{k\rightarrow \infty} \frac{k\,\tilde{u}_{k-1}}{\tilde{u}_k}(0)=\lambda_1(B^w_R) \quad .
\end{equation}

We now need the following

\begin{lemma}\label{C2limit}
Let $g_k(r):=\tilde{u}_k(r)/\tilde{u}_k(0)$ be the functions
defined by \eqref{eq_ukint} normalized so that $g_k(0)=1$ for all
$k$. Then $g_\infty(r):=\lim_{k\rightarrow \infty} g_k(r)$ is a
positive radial $C^2$-function defined on $B^w_R$ such that $\Delta g_\infty(r)= -\lambda g_\infty(r)$ and
$g_\infty(R)=0$.
\end{lemma}
\begin{proof}
Since $k \tilde{u}_{k-1}(r)/\tilde{u}_k(r)$ is increasing in $r$ for all $k$,
we have that
\[\frac{k\, \tilde{u}_{k-1}}{\tilde{u}_{k}}(r)\geq \frac{k\, \tilde{u}_{k-1}}{\tilde{u}_k}(0) \quad ,  \]
and then, for a fixed $r$, $\{g_k(r)\}_{k=1}^\infty$ is a decreasing sequence of bounded
functions converging  pointwise to a function $g_\infty(r)$.
Moreover, since $\tilde{u}_{k-1}$ is a decreasing function
\begin{equation}\label{g'kbounded}
|g'_k(r)|=\frac{k\,\int_0^r w^{m-1}(s)\tilde{u}_{k-1}(s)\,ds}{
w^{m-1}(s)\tilde{u}_k(0)}\leq
\frac{k\,\tilde{u}_{k-1}(0)}{\tilde{u}_k(0)} q_w(r)\leq \lambda\,
\max_{[0,R]} \{q_w(r)\},
\end{equation}
where we have used that $\{k\,
\tilde{u}_{k-1}(0)/\tilde{u}_k(0)\}_k$ is an increasing sequence
converging to $\lambda$ and that $q_w(r)$ is the isoperimetric
quotient defined in \eqref{isoquotient} that is a continuous
function on the compact interval $[0,R]$.

Then, the first derivatives of the functions $g_k$ are uniformly
bounded, and hence the sequence of functions is uniformly bounded and equicontinuous. As a
consequence of the Ascoli-Arzela Theorem, $\{g_k\}_{k=1}^\infty$ converges to
$g_\infty$ uniformly and $g_\infty$ is a continuous function. In
addition,
\begin{equation}
\begin{aligned}
\lim_{k\rightarrow \infty} g'_k(r) &= \lim_{k\rightarrow \infty
}\frac{-k\,\int_0^r w^{m-1}(s)\tilde{u}_{k-1}(s)\,ds}{w^{m-1}(r)u_k(0)}\\
&=\lim_{k\rightarrow \infty}
\frac{-k\,\tilde{u}_{k-1}(0)}{\tilde{u}_k(0)}\frac{\int_0^r
w^{m-1}(s)g_{k-1}(s)\,ds}{w^{m-1}(r)} =-\lambda\, \frac{\int_0^r
w^{m-1}(s)g_\infty(s)\,ds}{w^{m-1}(r)}.
\end{aligned}
\end{equation}
Now, taking into account that
\begin{displaymath}
\tilde{u}^{''}_k(r)=-k\,\tilde{u}_{k-1}(r)-(m-1)\frac{w'(r)}{w(r)}\,\tilde{u}^{'}_{k}(r) \quad ,
\end{displaymath}
and using \eqref{g'kbounded}
and that $g_{k-1}$ is bounded above by $1$,
\begin{equation}
\begin{aligned}
|g_k''(r)|&=|-\frac{k\,\tilde{u}_{k-1}(0)}{\tilde{u}_k(0)}\,g_{k-1}(r)-(m-1)\eta_w(r)g'_k(r)|\\
&\leq  \lambda|g_{k-1}(r)| +(m-1)\lambda |\eta_w(r)q_w(r)|\\
&\leq  \lambda +(m-1)\lambda |\eta_w(r)q_w(r)|.
\end{aligned}
\end{equation}

But $w(r)$ satisfies $w(0)=0$ and $w'(0)=1$. Hence,
 $\eta_w(r) q_w(r)=\frac{w'(r) \int_0^r w^{m-1}(t) dt}{w^{m}(r)}$ is a continuous function defined on
$[0,R]$ because by L'Hopital rule
\begin{equation}
\lim_{r \to 0} \frac{\int_0^r w^{m-1}(t) dt}{w^{m}(r)}=\frac{1}{m} \quad .
\end{equation}

 Therefore, the derivatives $g_k''(r)$ are uniformly bounded on
$[0,R]$ and the functions $g_k'$ converge uniformly. As a
consequence of this,
\[g_\infty'(r)=\lim_{k\rightarrow \infty} g_k'(r)=-\lambda\, \frac{\int_0^r
w^{m-1}(s)g_\infty(s)\,ds}{w^{m-1}(r)}.\]

To finish the proof we only need to show that $g''_k(r)$ also converges
uniformly to $g''_\infty(r)$. But it is clear using the definition of uniform convergence, since

\begin{eqnarray*}
|g_k''(r)&+&\lambda g_\infty(r) +(m-1)\eta_w(r)g_\infty'(r)|=\\
 &&
|-\frac{k\,\tilde{u}_{k-1}(0)}{\tilde{u}_k(0)}g_{k-1}(r)-(m-1)\eta_w(r)g_k'(r)+\lambda
g_\infty(r) +(m-1)\eta_w(r)g_\infty'(r)|\\
&\leq & |\lambda
g_\infty(r)-\frac{k\,\tilde{u}_{k-1}(0)}{\tilde{u}_k(0)}g_{k-1}(r)|\\
&&+(m-1)\Big|\eta_w(r)\frac{\int_0^r
w^{m-1}(s)\Big(\frac{k\,\tilde{u}_{k-1}(0)}{\tilde{u}_k(0)}\,g_{k-1}(s)-\lambda\,g_\infty(s)\Big)\,ds}{w^{m-1}(r)}|\\
&\leq &
\sup_{[0,R]}\{|\lambda\,g_\infty(r)-\frac{k\,\tilde{u}_{k-1}(0)}{\tilde{u}_k(0)}\,g_{k-1}(r)|\}(1+(m-1)|\eta_w(r)\,q_w(r)|) \quad .
\end{eqnarray*}
We use the fact that the function $k\,\tilde{u}_{k-1}(0)/\tilde{u}_k(0)\, g_{k-1}(r)$
converges uniformly to $\lambda\,g_\infty(r)$ and that
$\eta_w(r)\,q_w(r)$ is a continuous function defined on $[0,R]$.
Then,
\begin{equation}\label{eq_g''}
g''_\infty(r)=-\lambda g_\infty(r) -(m-1)\eta_w(r)g_\infty'(r).
\end{equation}
The proof of the lemma is then finished by observing that $g_k$ are $C^2$-functions and that
$g_k(R)=0$ for all $R$.
\end{proof}

Returning to the proof of the Theorem, to show equality
\begin{displaymath}
\lambda_1(B^w_R)=\lim_{k\rightarrow
\infty}\frac{k\,\mathcal{A}_{k-1}(B^w_R)}{\mathcal{A}_{k}(B^w_R)},
\end{displaymath}
 we argue in the same way, applying again Barta's Lemma using a positive $C^2$-function satisfying that $ \Delta h_\infty(r)= -\mu h_\infty(r)$ and $h_\infty(R)=0$.




 We need the following

 \begin{lemma}\label{C3limit}
Let us define the functions $h_k(r)=\tilde{u}_k(r)/(-\tilde{u}'_k(R))$ for all $k \geq 0$. Then $h_\infty(r) =\lim_{k\rightarrow \infty} h_k(r)$ is a
positive radial $C^2$-function defined on $B^w_R$ such that $\Delta h_\infty(r)= -\mu h_\infty(r)$ and
$h_\infty(R)=0$.
\end{lemma}
\begin{proof}
 Let us normalize the functions $\tilde{u}_k$ by putting
$h_k(r)=\tilde{u}_k(r)/(-\tilde{u}'_k(R))$. Then, since $k
\tilde{u}_{k-1}/\tilde{u}_k$ is increasing for all $k$, we have
that
\[\frac{k\, \tilde{u}_{k-1}}{\tilde{u}_{k}}(r)\leq \frac{k\, \tilde{u}_{k-1}}{\tilde{u}_k}(R)=\frac{k\,\tilde{u}'_{k-1}(R)}{\tilde{u}'_k(R)} \quad ,  \]

\noindent and $\{h_k(r)\}_{k=1}^\infty$ is an increasing sequence of functions. Moreover,
$h_k(r)\leq h_k(0)=\tilde{u}_k(0)/(-\tilde{u}_k'(R))$ and

\begin{eqnarray}\label{h_k}
\frac{-\tilde{u}_k'(R)}{\tilde{u}_k(0)}&=&\frac{k\,
\tilde{u}_{k-1}(0)}{\tilde{u}_k(0)}\frac{\int_0^R w^{m-1}(s)
\frac{\tilde{u}_{k-1}(s)}{\tilde{u}_{k-1}(0)}\,ds}{w^{m-1}(R)}\\
\nonumber &\geq& \frac{1}{\tilde{u}_1(0)}\frac{\int_0^R w^{m-1}(s)
g_\infty(s)\,ds}{w^{m-1}(R)}=\frac{-g_\infty'(R)}{\lambda_1\,\tilde{u}_1(0)}.
\end{eqnarray}

The functions $h_k$ are therefore bounded from above and thence they
converge pointwise to a function $h_\infty$. Following the lines
of the proof of lemma above and taking into account that
$h_k(r)=\frac{\tilde{u}_k(0)}{-\tilde{u}_k'(R)}\,g_k(r)$ and that
$\frac{-\tilde{u}_k'(R)}{\tilde{u}_k(0)}$ converges to
$g_\infty'(R)$ by \eqref{h_k}, we can conclude that $h_\infty$ is
a $C^2$-function satisfying that $h_\infty(R)=0$ and
\begin{equation}\label{h''k}
h_\infty''(r)=-\mu h_\infty(r)-(m-1)\eta_w(r)h_\infty'(r).
\end{equation}

This concludes the proof of Lemma \ref{C3limit}.
\end{proof}

To finish the proof of Theorem  \ref{lambda1model} we observe that, on the other hand,
$h_\infty(r)=\lim_{k\rightarrow \infty}
\frac{\tilde{u}_k(R)}{-\tilde{u}_k'(R)}$ is a radial
$C^2$-function on $B^w_R$ satisfying that $h_\infty(R)=0$.
Moreover, $\Delta h_\infty=-\mu\, h_\infty$ by (\ref{h''k}), where

\begin{displaymath} \mu:=\lim_{k\,\rightarrow \infty}
\frac{k\mathcal{A}_{k-1}(B^w_R)}{\mathcal{A}_{k}(B^w_R)}.
\end{displaymath}
Again, Barta's inequalities give us the result.
\end{proof}

\section{Extrinsic and intrinsic Eigenvalue comparison theorems}\label{comparisons}

In this section we now state and prove the comparison theorems for the
first eigenvalue of extrinsic balls of submanifolds that we alluded to in the introduction, but here in the more general contexts of the  comparison constellations. We divide this section into three parts.

In the first we use Theorem \ref{lambda1model} to get upper and lower bounds for the first Dirichlet eigenvalues of the extrinsic balls of a submanifold.

In the second subsection we present the comparison between the
first Dirichlet eigenvalues of the extrinsic balls and the
geodesic balls with the same radius in the spaces used as a model
and described in  Subsection \ref{subsec1}, but now the strategy
is based on the description of the first Dirichlet eigenvalue of a
domain in a Riemannian manifold given by McDonald and Myers in
\cite{McM}. With this method it is furthermore possible to
describe the equality case and obtain a rigidity result in one of
the two settings studied.

The third subsection is devoted to an investigation of the
intrinsic case, namely to apply the results obtained in
Subsections \ref{subsec1} and \ref{subsec2} to the intrinsic
geodesic balls of a manifold with sectional curvatures bounded
from above or from below by the corresponding sectional curvatures
of rotationally symmetric model spaces. In both of these
(intrinsic) cases we also obtain the corresponding rigidity
results.

\subsection{Extrinsic comparison. First strategy}\label{subsec1}

\begin{theorem}\label{th_const_below}
Let $\{ N^{n}, P^{m}, C_{w,g,h}^{m} \}$ denote a comparison
constellation boun\-ded from below in the sense of Definition
\ref{defConstellatNew1}. Assume that $M^m_W=C_{w,g,h}^{m}$ is
$w$-balanced from below. Let $D_R$ be a smooth precompact
extrinsic $R$-ball in $P^m$
 with center at a point $p \in P \subset N$ which also serves as a pole in
 $N$.

 Then we have the following inequalities:
 \begin{equation}\label{ineqleq_submanifold}
 \lambda_1(D_R) \leq \lambda_1 (B^W_{s(R)})=\lim_{k\,\rightarrow \infty}
\frac{k\mathcal{A}_{k-1}(B^W_{s(R)})}{\mathcal{A}_{k}(B^W_{S(R)})} \leq \frac{\Vol(B^W_{S(R)})}{\mathcal{A}_{1}(B^W_{S(R)})} \quad ,
\end{equation}
 where $B^W_{s(R)}$ is the geodesic ball in $M^m_W$.
\end{theorem}
 \begin{proof}
 Let $g_\infty(s)$ be the eigenfunction of $\lambda_1(B^W_{s(R)})$
 obtained in Theorem \ref{lambda1model} and given by

 \[g_\infty(s)=\lim_{k\rightarrow \infty} \frac{\tilde{u}^W_k(s)}{\tilde{u}^W_k(0)}.\]

 If we compose $g_\infty$ with the stretching function defined
 by \eqref{stretchingfunct}, we obtain a function $\tilde{g}:[0,R]\rightarrow
 \mathbb{R}$ as
\[\tilde{g}(r)=\lim_{k\rightarrow \infty} \frac{\tilde{u}^W_k(s(r))}{\tilde{u}^W_k(0)}=
\lim_{k\rightarrow \infty} \frac{f_k(r)}{f_k(0)},\] where $f_k$
are the functions introduced in Lemma \ref{paren}.

Now let $r$ denote the smooth distance to the pole $p$ on $N$. We
define the radial function $v: D_R\rightarrow \mathbb{R}$ by
$v(q)=\tilde{g}(r(q))$. Using Theorem \ref{corLapComp} and the
fact that
\begin{displaymath}
\tilde{g}'(r)=\lim_{k \rightarrow \infty}
\frac{f_k'(r)}{f_k(0)}\leq 0,
\end{displaymath}
since $f_k'\leq 0$ for all $k$ and the convergence is uniform, we
have that
\begin{eqnarray*}
\Delta^P v=\Delta^P (\tilde{g}\circ r)\geq (\tilde{g}''(r)-
\tilde{g}'(r)\eta_w(r))\|\nabla^P r\|^2+ m
\tilde{g}'(r)(\eta_w(r)+\langle \nabla^N r,H_P\rangle)
\end{eqnarray*}
Again by the uniform convergence proved in Lemma \ref{C2limit},
\begin{displaymath}
\tilde{g}''(r)- \tilde{g}'(r)\eta_w(r)=\lim_{k\rightarrow \infty}
\frac{f_k''(r)-\eta_w(r)f_k'(r)}{f_k(0)}\geq 0,
\end{displaymath}
by virtue of Lemma \ref{paren}.

Then, if we interchange the limits and the derivatives and use
Lemma \ref{paren},
\begin{eqnarray*}
\Delta^P v=\Delta^P (\tilde{g}\circ r)&\geq& (\tilde{g}''(r)-
\tilde{g}'(r)\eta_w(r))g(r)^2+ m \tilde{g}'(r)(\eta_w(r)-h(r))\\
&=&\lim_{k\rightarrow \infty} \frac{-k\, f_{k-1}(r)}{f_k(0)}=
\lim_{k\rightarrow \infty} \frac{-k\,
f_{k-1}(0)}{f_k(0)}\frac{f_{k-1}(r)}{f_{k-1}(0)}\\
&=&-\lambda_1(B^W_{s(R)})\,\tilde{g}\circ
r=-\lambda_1(B^W_{s(R)})\,v.
\end{eqnarray*}

Therefore,
\begin{displaymath}
\sup_{D_R}\{-\Delta^P v/v\}\leq \lambda_1(B^W_{s(R)}).
\end{displaymath}
Since $v_{\vert \partial D_r}=0$, Barta's inequalities give us the
result.
\end{proof}



\begin{theorem} \label{th_const_above}
Let $\{ N^{n}, P^{m}, C_{w,1,h}^{m} \}$  denote a comparison
constellation bounded from above. Assume that
$M^m_W=C_{w,1,h}^{m}$ is $w$-balanced from below. Let $D_R$ be a
smooth precompact extrinsic $R$-ball in $P^m$,
 with center at a point $p \in P$ which also serves as a pole in $N$.
 Then we have
 \begin{equation}\label{ineqgeq_submanifold1}
 \lambda_1(D_R) \geq \lambda_1 (B^W_{R}),
 \end{equation}
 Suppose moreover that
 \begin{equation}
 (m-1)\cdot\inf_{r \in [0,R]}\eta_w(r) \geq m\cdot \sup_{r \in [0,R]}h(r)\quad .
 \end{equation}
 Then 
 \begin{equation}\label{ineqgeq_submanifold2}
 \lambda_1 (B^W_{R})\geq \frac{1}{4}\left( (m-1)\cdot \left(\inf_{r \in [0,R]}\eta_w(r)\right) -m\cdot \left(\sup_{r \in [0,R]}h(r)\right) \right)^2,
 \end{equation}
 where $B^W_{R}(\widetilde{p})$ is the pole centered geodesic ball in the model space $M^m_W$.

\end{theorem}

 \begin{proof}
 The proof follows the line of the proof of Theorem \ref{th_const_below}. Let $g_\infty(r)$ be the eigenfunction of $\lambda_1(B^W_{R})$
 obtained in Theorem \ref{lambda1model} and given by

 \[g_\infty(r)=\lim_{k\rightarrow \infty} \frac{\tilde{u}^W_k(r)}{\tilde{u}^W_k(0)}.\]

 In this case the stretching function is the identity, so we transplant $g_\infty$ to $D_R$ composing it with the smooth distance to the pole
$p$ on $N$. We obtain the radial function $v: D_R\rightarrow
\mathbb{R}$ given by $v(q)=g_\infty(r(q))$. As before, it is easy
to check that

\begin{eqnarray*}
\Delta^P v=\Delta^P (g_\infty\circ r)&\leq& g_\infty''(r)-
g_\infty'(r)\eta_w(r)+ m g_\infty'(r)(\eta_w(r)-h(r))\\
&=&-\lambda_1(B^W_{R})\,g_\infty\circ r=-\lambda_1(B^W_{R})\,v.
\end{eqnarray*}

Therefore,
\begin{displaymath}
\inf_{D_R}\{-\Delta^P v/v\}\geq \lambda_1(B^W_{R}),
\end{displaymath}
and since $v_{\vert \partial D_r}=0$, Barta's inequalities again give us
the first inequality.

To prove inequality (\ref{ineqgeq_submanifold2}), we
follow the argument in Theorem 2 in \cite{CheLe} -- see also the
proof of Lemma 2.3 in \cite{BM} -- and use the variational
Rayleigh formulation of the first Dirichlet eigenvalue of the
pole-centered geodesic ball $B^W_{R}(\widetilde{p})$. Given a
smooth function $f \in C^{\infty}_0(B^W_{R}(\widetilde{p}))$, let
us consider the vector field $X =\nabla^{M^m_W} \mu(r)$, where
$\mu(r)$ is any smooth radial function. Following the lines of
\cite{CheLe} and computing $\Div (f^2\,X)$, we then have for any
$\varepsilon
>0$ (a parameter that we shall fix later):
\begin{equation}\label{eqmu}
\int_{B^W_R} |\mu'(r)| \Vert \nabla^{M^m_W}  f\,\Vert^2 d\sigma
\geq \int_{B^W_R} f^2\, \left(\varepsilon \Delta^{M^m_W}\mu(r) -
\varepsilon^2 |\mu'(r)|\right) d\sigma \quad .
\end{equation}

We consider $\mu(r)=r$.  The vector field $X =\nabla^{M^m_W}
\mu(r)$ is not necessarily smooth, but we may then approximate it
by the smooth field $X = \nabla^{M^m_W} r^{\alpha}$, $1 < \alpha <
2 $, in the punctured ball $B^W_{R}(\widetilde{p}) -
B^W_{\delta}(\widetilde{p})$ and recover the same conclusion as
below following the argument in  \cite[p. 286--288]{BM} by letting
$\alpha \to 1$ and $\delta \to 0$. It is straightforward to see,
using the definition of isoperimetric comparison space, Definition
\ref{defCspace}, that
\begin{equation}\label{eqlap}
\begin{aligned}
\Delta^{M^m_W} r&=(m-1)\eta_W'(r)=(m-1)\eta_w(r)-mh(r) \\ &\geq (m-1) \inf_{r \in [0,R]}\eta_w(r)-m\,\sup_{r \in [0,R]} h(r) \quad .
\end{aligned}
\end{equation}
Using this inequality  in equation (\ref{eqmu}) with $\mu(r)=r$, we obtain the following inequality, where $L_R:= (m-1) \inf_{r \in [0,R]}\eta_w(r)-m\,\sup_{r \in [0,R]} h(r)$ and $Q(\varepsilon):=\varepsilon L_R-\varepsilon^2$:

\begin{equation}\label{eqL}
\int_{B^W_R}  \Vert \nabla^{M^m_W}  f \, \Vert^2 d\sigma \geq \int_{B^W_R} f^2 \, \left(\varepsilon L_R -\varepsilon^2 \right) d\sigma \quad .
\end{equation}

Since $Q(\varepsilon)$ attains its maximum at $\varepsilon_0=\frac{L_R}{2}$, with $Q(\varepsilon_0)=\frac{L^2_R}{4}$,
we finally obtain the desired Poincar\'{e} inequality for all $f \in C^{\infty}_0(B^W_{R}(\widetilde{p}))$:
\begin{equation}
\int_{B^W_R}  \Vert \nabla^{M^m_W}  f \, \Vert^2 d\sigma \geq \frac{L^2_R}{4}\int_{B^W_R} f^2 \, d\sigma \quad,
\end{equation}
which gives the Rayleigh quotient inequality for the first eigenvalue of $B^W_{R}(\widetilde{p})$:
\begin{equation}
\lambda_1 (B^W_{R})\geq \frac{L^2_R}{4} \quad .
\end{equation}
\end{proof}

\subsection{Extrinsic comparison. Second strategy}\label{subsec2}

We present an alternative proof of the first inequality in (\ref{ineqleq_submanifold}) from Theorem \ref{th_const_below} using Theorem \ref{thmMcDonaldExpress} and previous results from \cite{HMP2}:

\begin{theorem}[Second proof of the first inequality of (\ref{ineqleq_submanifold})]\label{th_const_below2}
Let $\{ N^{n}, P^{m}, C_{w,g,h}^{m} \}$ denote a comparison
constellation boun\-ded from below in the sense of Definition
\ref{defConstellatNew1}. Assume that $M^m_W=C_{w,g,h}^{m}$ is
$w$-balanced from below. Let $D_R$ be a smoothly precompact
extrinsic $R$-ball in $P^m$,
 with center at a point $p \in P$ which also serves as a pole in
 $N$.

 Then,
 \begin{equation}\label{ineqleq_submanifold2}
 \lambda_1(D_R) \leq \lambda_1 (B^W_{s(R)})\end{equation}
 where $B^W_{s(R)}$ is the geodesic ball in $M^m_W$.
\end{theorem}
\begin{proof}
We here show how to obtain inequality $\lambda_1(D_R) \leq \lambda_1 (B^W_{s(R)})$ using the  description of the first Dirichlet eigenvalue of a smooth precompact domain $D$ in a Riemannian manifold
given by P. McDonald and R. Meyers in \cite{McM}. When $D=D_R$, we have
\begin{equation}\label{desc}
\lambda_1(D_R)= \sup \left\{\eta \geq 0 \,:\, \lim_{n \to \infty}
\sup
\left(\frac{\eta}{2}\right)^n\frac{\mathcal{A}_{n}(D_R)}{\Gamma(n+1)}
<\infty\right\} \quad .
\end{equation}

On the other hand, we have, under the assumptions of the theorem, the following inequalities concerning  the mean exit time moment spectrum of the extrinsic balls, see \cite{HMP2}:

\begin{equation}\label{ineqspec}
\frac{\mathcal{A}_{n}(D_R)}{\Vol(\partial D_R)} \geq \frac{\mathcal{A}_{n}(B^W_{s(R)})}{\Vol(\partial B^W_{s(R)})}\,\,\,\,\,\,\textrm{for all}\,\, n \in \ene \quad .
\end{equation}

And moreover, the following inequalities concerning its volume, see \cite{HMP}:

\begin{equation}\label{ineqvol}
\begin{aligned}
\frac{\Vol(\partial D_R)}{\Vol(D_R)} &\leq \frac{\Vol(\partial B^W_{s(R)})}{\Vol( B^W_{s(R)})} \quad ,\\
\Vol(D_R) &\leq \Vol( B^W_{s(R)}) \quad .
\end{aligned}
\end{equation}

Then, using inequality (\ref{ineqspec}) the set
$$H_1:= \{\eta \geq 0 \,:\, \lim_{n \to \infty} \sup \left(\frac{\eta}{2}\right)^n\frac{\mathcal{A}_{n}(D_R)}{\Gamma(n+1)} <\infty\}$$
 is included in the set
 $$H_2:=\{\eta \geq 0 \,:\, \lim_{n \to \infty}  \sup \left(\frac{\eta}{2}\right)^n\frac{\mathcal{A}_{n}(B^W_{s(R)})}{\Gamma(n+1)} \frac{\Vol(\partial D_R)}{\Vol(\partial B^W_{s(R)})}<\infty\} \quad ,$$
 so we have
\begin{equation}
\begin{aligned}
\lambda_1(D_R)= \sup \{\eta \geq 0 \,:\, \lim_{n \to \infty} & \sup \left(\frac{\eta}{2}\right)^n\frac{\mathcal{A}_{n}(D_R)}{\Gamma(n+1)} <\infty\} \\ \leq
\sup \{\eta \geq 0 \,:\, \lim_{n \to \infty} & \sup \left(\frac{\eta}{2}\right)^n\frac{\mathcal{A}_{n}(B^W_{s(R)})}{\Gamma(n+1)} \frac{\Vol(\partial D_R)}{\Vol(\partial B^W_{s(R)})}<\infty\} \\= \frac{\Vol(\partial D_R)}{\Vol(\partial B^W_{s(R)})} \lambda_1(B^W_{s(R)}) \quad .
\end{aligned}
\end{equation}
Inequalities \eqref{ineqvol} give the desired inequality.
\end{proof}

In the same vein we also present the corresponding alternative proof of the eigenvalue comparison inequality  (\ref{ineqgeq_submanifold1}) from Theorem \ref{th_const_above} again using Theorem \ref{thmMcDonaldExpress} and \cite{HMP2}:

\begin{theorem}[Second proof of inequality (\ref{ineqgeq_submanifold1})] \label{th_const_above2}
Let $\{ N^{n}, P^{m}, C_{w,1,h}^{m} \}$  denote a comparison
constellation bounded from above. Assume that
$M^m_W=C_{w,1,h}^{m}$ is $w$-balanced from below. Let $D_R$ be a
smooth precompact extrinsic $R$-ball in $P^m$,
 with center at a point $p \in P$ which also serves as a pole in $N$.
 Then we have
 \begin{equation}\label{ineqgeq_submanifold3}
 \lambda_1(D_R) \geq \lambda_1 (B^W_{R}),
 \end{equation}
 where $B^W_{R}(\widetilde{p})$ is the pole centered geodesic ball in the model space $M^m_W$.

 If $M^m_W$ is strictly balanced (in the sense that inequality (\ref{eqBalA}) is an strict inequality for all radius $r$) equality in
 \eqref{ineqgeq_submanifold3}  for some fixed radius $R_0$ implies that $D_{R_0}$ is a geodesic cone in $N$. If $P$ is minimal and $N=\hh^n(b)$ is the hyperbolic space, we then have that $P$ is a totally
 geodesic submanifold in $\hh^n(b)$. 

\end{theorem}
\begin{proof}
As in the above theorem, we can again obtain inequality $\lambda_1(D_R) \geq \lambda_1 (B^W_{R})$ using the  description (\ref{desc}) of the first Dirichlet eigenvalue given in \cite{McM} restricted to the extrinsic balls.
Under the assumptions of the theorem we now have the following inequalities concerning the mean exit time moment spectrum of the extrinsic balls, see \cite{HMP2}:

\begin{equation}\label{ineqspec2}
\frac{\mathcal{A}_{n}(D_R)}{\Vol(\partial D_R)} \leq
\frac{\mathcal{A}_{n}(B^W_{R})}{\Vol(\partial
B^W_{R})}\,\,\,\,\textrm{for all}\,\, n \in \ene \quad .
\end{equation}

And the corresponding inequalities concerning the relative volumes, see \cite{HMP}:

\begin{equation}\label{ineqvol2}
\begin{aligned}
\frac{\Vol(\partial D_R)}{\Vol(D_R)} &\geq \frac{\Vol(\partial B^W_{R})}{\Vol( B^W_{R})} \quad , \\
\Vol(D_R) &\geq \Vol( B^W_{R})\quad .
\end{aligned}
\end{equation}

Then, using inequality (\ref{ineqspec2}) the set $$H_2:=\{\eta \geq 0 \,:\, \lim_{n \to \infty}  \sup \left(\frac{\eta}{2}\right)^n\frac{\mathcal{A}_{n}(B^W_{R})}{\Gamma(n+1)} \frac{\Vol(\partial D_R)}{\Vol(\partial B^W_{R})}<\infty\}$$  is included in the set $$H_1:= \{\eta \geq 0 \,:\, \lim_{n \to \infty} \sup \left(\frac{\eta}{2}\right)^n\frac{\mathcal{A}_{n}(D_R)}{\Gamma(n+1)} <\infty\}\quad ,$$ so we have

\begin{equation}
\begin{aligned}
\lambda_1(D_R)= \sup \{\eta \geq 0 \,:\, \lim_{n \to \infty} & \sup \left(\frac{\eta}{2}\right)^n\frac{\mathcal{A}_{n}(D_R)}{\Gamma(n+1)} <\infty\} \\ \geq
\sup \{\eta \geq 0 \,:\, \lim_{n \to \infty} & \sup \left(\frac{\eta}{2}\right)^n\frac{\mathcal{A}_{n}(B^W_{s(R)})}{\Gamma(n+1)} \frac{\Vol(\partial D_R)}{\Vol(\partial B^W_{R})}<\infty\} \\= \frac{\Vol(\partial D_R)}{\Vol(\partial B^W_{R})} \lambda_1(B^W_{R}) \quad .
\end{aligned}
\end{equation}

Then we have inequality $\lambda_1(D_R) \geq \lambda_1 (B^W_{R})$ using now inequalities (\ref{ineqvol2}).

This alternative proof allows us to discuss the equality case
as in the paper \cite{HMP2}: If $M^m_W$ is strictly balanced, (in the sense that inequality (\ref{eqBalA}) is an strict inequality for all radius $r$), equality $\lambda_1(D_{R_0}) = \lambda_1 (B^W_{R_0})$ for some fixed radius $R_0$ implies the equality between the volumes $ \Vol(D_{R_0}) = \Vol( B^W_{R_0})$, which implies in itself that $\nabla^P r=\nabla^N r$ and $D_{R_0}$ is a geodesic cone in $N$ swept out by the radial geodesics from the center $p$ of the ball. If $P$ is minimal and $N=\hh^n(b)$ is the hyperbolic space, we then have that $P$ is a totally geodesic submanifold in $\hh^n(b)$.
\end{proof}

\subsection{Intrinsic comparison}\label{subsec3}

We finally consider the intrinsic consequences of Theorems
\ref{th_const_below} and \ref{th_const_above} assuming that
$P^m=N^n$. In this case, the extrinsic distance to the pole $p$
becomes the intrinsic distance in $N$, so, for all $r$ the
extrinsic domains $D_r$ become the geodesic balls $B^N_r$ of the
ambient manifold $N$. Then, for all $x \in P$
\begin{eqnarray*}
\nabla^P r(x)&=&\nabla r (x),\\
H_P(x)&=&0.
\end{eqnarray*}
As a consequence, $\|\nabla^P r\|=1$, so $g(r(x))=1$ and
$\mathcal{C}(x)=h(r(x))=0$, the stretching function becomes the
identity $s(r)=r$, $W(s(r))=w(r)$, and the isoperimetric
comparison space $C_{w,g,h}^m$ is reduced to the auxiliary model
space $M^m_w$.\\

As a corollary of the proofs of the theorems above we therefore have:

\begin{theorem} \label{th_const_below_intrinsic}
Let $B^N_R$ be a geodesic ball of a complete Riemannian manifold
$N^n$ with a pole $p$ and suppose that the $p$-radial sectional
curvatures of $N^n$ are bounded from below by the $p_w$-radial
sectional curvatures of a $w$-model space $M^n_w$.  Then
\begin{equation}\label{ineqleq_intrinsic}
 \lambda_1(B^N_R) \leq \lambda_1 (B^w_{R}),
 \end{equation}
where $B^{w}_{R}$ is the $p_{w}$-centered geodesic ball in $M^n_w$. If $M^m_w$ is strictly balanced, then  equality in
 \eqref{ineqleq_intrinsic}  for some fixed radius $R_0$ implies that $B^N_{R_0}$ and  $B^w_{R_0}$ are isometric.
\end{theorem}

\begin{theorem} \label{th_const_above_intrinsic}
Let $B^N_R$ be a geodesic ball of a complete Riemannian manifold
$N^n$ with a pole $p$ and suppose that the $p$-radial sectional
curvatures of $N^n$ are bounded from above by the $p_w$-radial
sectional curvatures of a $w$-model space $M^n_w$. Then
\begin{equation}\label{ineqgeq_intrinsic}
 \lambda_1(B^N_R) \geq \lambda_1 (B^w_{R}),
 \end{equation}
where $B^{w}_{R}$ is the geodesic ball in $M^n_w$. If $M^m_w$ is strictly balanced, then  equality in
 \eqref{ineqgeq_intrinsic}  for some fixed radius $R_0$ implies that $B^N_{R_0}$ and  $B^w_{R_0}$ are isometric.

\end{theorem}

\begin{remark}
The equality statements -- now in \emph{both} theorems -- come from the fact that equality for some fixed radius $R_0$ in inequalities (\ref{ineqleq_intrinsic}) and (\ref{ineqgeq_intrinsic}), leads via arguments in Theorems \ref{th_const_below2} and \ref{th_const_above2} to the equality among the volumes of the geodesic balls $B^N_{R_0}$ and $B^w_{R_0}$. We conclude that they are therefore isometric using a direct application of Bishop's volume comparison theorem, as found e.g. in Corollary 3.2, Chapter IV in \cite{S}.
\end{remark}


\end{document}